\documentclass[a4,11pt]{article}
\usepackage[dvips]{graphicx}
\usepackage{amssymb}
\usepackage{amstext}
\usepackage{amsmath}
\usepackage{amscd}
\usepackage{amsthm}
\usepackage{amsfonts}
\usepackage{enumerate}
\usepackage{graphicx}
\usepackage{latexsym}
\usepackage{mathrsfs}
\usepackage{mathtools}
\usepackage{cases}
\usepackage[svgnames]{xcolor}
\usepackage{verbatim}
\usepackage{tikz}
\newtheorem{thm}{Theorem}[section]

\newtheorem{lem}[thm]{Lemma}

\newtheorem{ass}[thm]{Assumption}
\theoremstyle{definition}

\newtheorem{rem}[thm]{Remark}

\newtheorem*{claim*}{Claim}
\theoremstyle{remark}

\numberwithin{equation}{section}

\title{\vspace{-3cm}\textbf{A modification of the factorization method for scatterers with different physical properties}}

\author{Takashi FURUYA}

\date{}
\begin{document}
\maketitle
\begin{abstract}
We study an inverse acoustic scattering problem by the Factorization Method when the unknown scatterer consists of two objects with different physical properties. Especially, we consider the following two cases: One is the case when each object has the different boundary condition, and the other one is when different penetrability. Our idea here is to modify the far field operator depending on the cases to avoid unnecessary a priori assumptions.
\end{abstract}
\section{Introduction}
Sampling methods are proposed for reconstruction of shape and location in inverse acoustic scattering problems. In the last twenty years, sampling methods such as the Linear Sampling method of Colton and Kress \cite{Colton and Kress}, the Singular Sources Method of Potthast \cite{Potthast}, the Factorization Method of Kirsch \cite{Kirsch}, have been introduced and intensively studied. As an advantage of these sampling methods, the numerical implementation are so simple and fast. However, as disadvantage of sampling methods except the Factorization Method, only sufficient conditions are given for the identification of unknown scatterers. To overcome this drawback, that is, to provide necessary and sufficient conditions, the Factorization Method was introduced and developed by a lot of researchers. 
\par
However, for rigorous justification of the original Factorization Method, we have to assume that the wave number of the incident wave is not an eigenvalue of the Laplacian on an obstacle with respect to the boundary condition of the scattering problem. Kirsch and Liu \cite{Kirsch and Liu} eliminated this problem for the case of a single obstacle by assuming that a small ball is in the interior of the unknown obstacle. They modified the original far field operator by adding the far field operator corresponding to a small ball so that the Factorization Method can be applied to it. On the other hands, in the case of a scatterer consisting of two objects with different physical properties, this problem has been still open. For recent works discussing this case, we refer to \cite{Anagnostopoulos and Charalambopoulos and Kleefeld, Bondarenko and Kirsch and Liu, Kirsch and Kleefeld, Kirsch and Liu3, Yang and Zhang and Zhang}.
\par
In this paper, we study the Factorization Method for a scatterer consisting of two objects with different physical properties. Especially, we consider the following two cases: One is the case when each object has the different boundary condition, and the other one is when different penetrability. For recent works discussing such a scatterer, we refer to \cite{Kirsch and Grinberg2, Kirsch and Liu2, Liu}. We remark that these works have to assume that the wave number of the incident wave is not an eigenvalue of the Laplacian on impenetrable obstacles included in a scatterer. Our aim of this paper is to eliminate this restriction by developing the idea of \cite{Kirsch and Liu}.
\par
We begin with the formulations of the scattering problems. Let $k>0$ be the wave number and for $\theta \in \mathbb{S}^{2}$ be incident direction. Here, $\mathbb{S}^{2}=\{x \in \mathbb{R}^3 : |x|=1 \}$ denotes the unit spherer in $\mathbb{R}^3$. We set
\begin{equation}
u^i(x):=\mathrm{e}^{ik\theta \cdot x}, \ x \in \mathbb{R}^3,\label{1.1}
\end{equation}
where {\it i} in the left hand side stands for {\it incident plane wave}. Let $\Omega \subset \mathbb{R}^3$ be a bounded open set and let its exterior $\mathbb{R}^3\setminus  \overline{\Omega}$ be connected. We assume that $\Omega$ consists of two bounded domains, i.e., $\Omega=\Omega_1 \cup \Omega_2$ such that 
$\overline{\Omega_1} \cap \overline{\Omega_2}=\emptyset$. 
We consider the following two cases.

{\bf The first case. $\Omega_1$ is an impenetrable obstacle with Dirichlet boundary condition, and $\Omega_2$ with Neumann boundary condition.}
Find $u^{s} \in H^{1}_{loc}(\mathbb{R}^3\setminus  \overline{\Omega})$ such that
\begin{equation}
\Delta u^{s}+k^2u^{s}=0 \ \mathrm{in} \ \mathbb{R}^3\setminus  \overline{\Omega}, \label{1.2}
\end{equation}
\begin{equation}
u^{s}=-u^{i} \ \mathrm{on} \ \partial{\Omega_1}, \label{1.3}
\end{equation}
\begin{equation}
\frac{\partial u^{s}}{\partial \nu_{\Omega_2}}=-\frac{\partial u^{i}}{\partial \nu_{\Omega_2}}\ \mathrm{on} \ \partial{\Omega_2}, \label{1.4}
\end{equation}
\begin{equation}
\lim_{r \to \infty}r \biggl( \frac{\partial u^{s}}{\partial  r}-iku^s \biggr)=0, \label{1.5}
\end{equation}
where $r=|x|$, and (\ref{1.5}) is the {\it Sommerfeld radiation condition}. Here, $H^{1}_{loc}(\mathbb{R}^3\setminus  \overline{\Omega})=\{u : \mathbb{R}^3\setminus \overline{\Omega} \to \mathbb{C} : u \bigl|_{B} \in H^{1}(B)\ \mathrm{for\ all\ open\ balls}\ B \}$ denotes the local Sobolov space of one order. $\nu_{\Omega_2}(x)$ denotes the unit normal vector at $x \in \partial \Omega_2$.
We refer to Theorem 7.15 in \cite{McLean} for the well posedness of the problem (\ref{1.2})--(\ref{1.5}), and refer to \cite{Kirsch and Grinberg2} and \cite{Liu} for the factorization method in this case.
\\

{\bf The second case. $\Omega_1$ is a penetrable medium modeled by a contrast function $q \in L^{\infty}(\Omega_1)$} (that is, $\Omega_1=\mathrm{supp}q$), {\bf and $\Omega_2$ is an impenetrable obstacle with Dirichlet boundary condition}.
Find $u^{s} \in H^{1}_{loc}(\mathbb{R}^3\setminus  \overline{\Omega_2})$ such that 
\begin{equation}
\Delta u^{s}+k^2(1+q)u^{s}=-k^2qu^{i} \ \mathrm{in} \ \mathbb{R}^3\setminus  \overline{\Omega_2}, \label{1.6}
\end{equation}
\begin{equation}
u^{s}=-u^{i} \ \mathrm{on} \ \partial{\Omega_2}, \label{1.7}
\end{equation}
\begin{equation}
\lim_{r \to \infty}r \biggl( \frac{\partial u^{s}}{\partial  r}-iku^s \biggr)=0. \label{1.8}
\end{equation}
Note that we extend $q$ by zero outside $\Omega_1$. The well posedness of the problem (\ref{1.6})--(\ref{1.8}) and its factorization method was shown in \cite{Kirsch and Liu2}.
\par
In both cases, it is well known that the scattered wave $u^{s}$ has the following asymptotic behavior:
\begin{equation}
u^s(x,\theta)=\frac{\mathrm{e}^{ik|x|}}{4 \pi |x|}u^{\infty}(\hat{x},\theta)+O\biggl(\frac{1}{|x|^{2}} \biggr), \ |x| \to \infty, \ \ \hat{x}:=\frac{x}{|x|}. \label{1.9}
\end{equation}
The function $u^{\infty}$ is called the far field pattern of $u^s$. With the far field pattern $u^{\infty}$, we define the far field operator $F :L^{2}(\mathbb{S}^{2}) \to L^{2}(\mathbb{S}^{2})$ by
\begin{equation}
Fg(\hat{x}):=\int_{\mathbb{S}^{2}}u^{\infty}(\hat{x},\theta)g(\theta)ds(\theta), \ \hat{x} \in \mathbb{S}^{2}. \label{1.10}
\end{equation}
We write the far field operator of the problem (\ref{1.2})--(\ref{1.5}) as $F=F^{Mix}_{\Omega_1,\Omega_2}$, and (\ref{1.6})--(\ref{1.8}) as $F=F^{Mix}_{\Omega_1q,\Omega_2}$ , respectively.
The inverse scattering problem we consider is to reconstruct  $\Omega$ from the far field pattern $u^{\infty}(\hat{x},\theta)$ for all $\hat{x},\theta \in \mathbb{S}^{2}$. In other words, given the far field operator $F$, reconstruct $\Omega$.
\par
Our contribution in this paper is, in both cases, to give the characterization of $\Omega_1$ without a priori assumptions for the wave number $k>0$. But we have to know the topological properties of $\Omega$. More precisely, an {\it inner} domain $B_1$ of $\Omega_1$ (based on \cite{Kirsch and Liu}), and an {\it outer} domain $B_2$ of $\Omega_2$ (\cite{Kirsch and Grinberg2}), have to be a priori known. Furthermore, we take an {\it additional} domain $B_3$ in the interior of $B_2$. By adding artificial far field operators corresponding to $B_1$, $B_2$, and $B_3$, we modify the original far field operator $F$.
\par
In the first case, we give the following characterization:
\begin{ass}
Let bounded domain $B_1$ and $B_2$ be a prior known. Assume that $\overline{B_1}\subset \Omega_1$, $\overline{\Omega_2}\subset B_2$, $\overline{\Omega_1} \cap \overline{B_2}=\emptyset$.
\end{ass}

\begin{figure}[h]
\centering
\begin{tikzpicture}[scale=0.7]
\draw [very thick](6,2) circle [x radius=2cm, y radius=1.4cm, rotate=0];
\fill [black!20](6,2) circle [x radius=1.2cm, y radius=1cm, rotate=130];
\fill [black!20](1.5,2) circle [x radius=1.4cm, y radius=10mm, rotate=100];
\node (A) at (1.6,1.2) [above] {{\large$\Omega_1$}} ;
\node (B) at (6,1.65) [above] {{\large$\Omega_2$}} ;
\node (C) at (8.3,2.8) [above] {{\large $B_2$}} ;
\node (H) at (5.9,1.1) [above] {\large Neumann} ;
\node (H) at (1.5,0.7) [above] {\large Dirichlet} ;
\draw [very thick](1.5,2.7) circle (0.5);
\node (E) at (1.5,2.3) [above] {{\large $B_1$}} ;
\draw [very thick](7,2.5) circle (0.5);
\node (E) at (7.05,2.1) [above] {{\large $B_3$}} ;
\end{tikzpicture}
\caption{}
\end{figure}

\begin{thm}\label{thm1.2} 
For $\hat{x} \in \mathbb{S}^{2}$, $z \in \mathbb{R}^3$, define
\begin{equation}
\phi_z(\hat{x}):=\mathrm{e}^{-ikz \cdot \hat{x}}. \label{1.11}
\end{equation}
Let $\mathrm{Assumption \ 1.1}$ hold. Take a positive number  $\lambda_0>0$, and a bounded domain $B_3$ with $\overline{B_3} \subset B_2$. $($See $\mathrm{Figure\ 1}$$.)$ Then, for $z \in \mathbb{R}^3 \setminus  \overline{B_2}$ \label{1.12}
\begin{equation}
z \in \Omega_1
\Longleftrightarrow
\sum_{n=1}^{\infty}\frac{|(\phi_z,\varphi_n)_{L^2(\mathbb{S}^{2})}|^2}{\lambda_n} < \infty, \label{1.12}
\end{equation}
where $(\lambda_n,\varphi_n)$ is a complete eigensystem of $F_{\#}$ given by
\begin{equation}
F_{\#}:=\bigl|\mathrm{Re}F\bigr|+\bigl|\mathrm{Im}F\bigr|,\label{1.13}
\end{equation}
where $F:=F^{Mix}_{\Omega_1,\Omega_2}+F^{Dir}_{B_2}+F^{Imp}_{B_1\cup B_3,i\lambda_0}$. Here, $F^{Dir}_{B_2}$ and $F^{Imp}_{B_1\cup B_3,i\lambda_0}$ are the far field operators for the pure Dirichlet boundary condition on $B_2$, and for the pure impedance boundary condition on $B_1 \cup B_3$ with an impedance function $i\lambda_0$, respectively.
\end{thm}
Latter, we explain artificial far field operators $F^{Dir}_{B_2}$ and $F^{Imp}_{B_1\cup B_3,i\lambda_0}$ in Section 2, and prove Theorem 1.2 in Section 3.
\par
In the second case, we give the following characterization:
\begin{ass}
Let a bounded domain $B_2$ be a priori known. Assume the following assumptions:
\begin{description}
  \item[(i)] $q \in L^{\infty}(\Omega_1)$ with $\mathrm{Im}q  \geq 0 \ in \ \Omega_1$.
  
  \item[(ii)] $|q|$ is locally bounded below in $\Omega_1$, i.e., for every compact subset $M \subset \Omega_1$, there exists $c>0$ (depend on $M$) such that $|q|\geq c \ \mathrm{in}\ M$. 

  \item[(iii)] $\overline{\Omega_2}\subset B_2$, $\overline{\Omega_1} \cap \overline{B_2}=\emptyset$.
  
   \item[(iv)] There exists $t \in (\pi/2, 3\pi/2)$ and $C>0$ such that
$\mathrm{Re}(\mathrm{e}^{-it}q)$ $\geq$ $C|q|$ a.e. in  $\Omega_1$.
\end{description}
\end{ass}

\begin{figure}[h]
\centering
\begin{tikzpicture}[scale=0.7]
\draw [very thick](6,2) circle [x radius=2.3cm, y radius=1.6cm, rotate=0];
\fill [black!20](6,2) circle [x radius=1.2cm, y radius=1cm, rotate=130];
\fill [black!20](1.5,2) circle [x radius=1.4cm, y radius=10mm, rotate=100];
\node (A) at (1.5,2.3) [above] {{\large$\Omega_1$}} ;
\node (B) at (6,1.8) [above] {{\large$\Omega_2$}} ;
\node (C) at (8.3,2.8) [above] {{\large $B_2$}} ;
\node (H) at (6,1.2) [above] {\large Obstacle} ;
\node (H) at (1.5,0.8) [above] {\large Medium} ;
\draw [very thick](7,2.5) circle (0.5);
\node (E) at (7.05,2.1) [above] {{\large $B_3$}} ;
\end{tikzpicture}
\caption{}
\end{figure}

\begin{thm}
Let $\mathrm{Assumption \ 1.3}$ hold.
Take a positive number $\lambda_0>0$, and a bounded domain $B_3$ with $\overline{B_3} \subset B_2$. $($See $\mathrm{Figure\ 2}$$.)$ Then, for $z \in \mathbb{R}^3 \setminus  \overline{B_2}$ 
\begin{equation}
z \in \Omega_1
\Longleftrightarrow
\sum_{n=1}^{\infty}\frac{|(\phi_z,\varphi_n)_{L^2(\mathbb{S}^{2})}|^2}{\lambda_n} < \infty, \label{1.15}
\end{equation}
where $(\lambda_n,\varphi_n)$ is a complete eigensystem of $F_{\#}$ given by
\begin{equation}
F_{\#}:=\bigl|\mathrm{Re}\bigl(\mathrm{e}^{-it} F\bigr)\bigr|+\bigl|\mathrm{Im}F\bigr|, \label{1.16}
\end{equation}
where $F:=F^{Mix}_{\Omega_1q,\Omega_2}+F^{Dir}_{B_2}+F^{Imp}_{B_3,i\lambda_0}$.
Here, the function $\phi_z$ is given by $($$\ref{1.11}$$)$.
\end{thm}
We prove Theorem 1.4 in Section. We can also give the characterization by replacing (iv) in Assumption 1.3 with
\begin{description}
\item[(iv')] {\it There exists $t \in [0, \pi/2) \cup (3\pi/2, 2 \pi]$ and $C>0$ such that
$\mathrm{Re}(\mathrm{e}^{-it}q)$ $\geq$ $C|q|$ a.e. in  $\Omega_1$.}
\end{description} 
For details, see Assumption 4.5 and Theorem 4.6.
\par
Let us compare our works (Theorems 1.2 and 1.4) with previous works from the mathematical point of view of a priori assumptions. For Theorem 1.2 we refer to Theorem 2.5 of \cite{Liu}, and for Theorems 1.4 we refer to Theorem 3.9 (b) of \cite{Kirsch and Liu2}. These previous works also gave the characterization of $\Omega_1$ by assuming the existence of outer domain $B_2$ of $\Omega_2$ and that the wave number $k^2$ is not an eigenvalue on an obstacle, while, in our work we can choose arbitrary wave number $k>0$ by introducing extra artificial domains such as $B_1$, $B_2$, and $B_3$, which are not so difficult topological assumptions. 
\par
This paper is organized as follows. In Section 2, we recall a factorization of the far field operator and its  properties. In Section 3 and Section 4, we prove Theorems 1.2 and 1.4, respectively.
\section{A factorization for the far field operator}
In Section 2, we briefly recall a factorization for the far field operators and its properties. 
\par

First, we consider a factorization of the far field operator for the pure boundary condition. Let $B$ be a bounded open set and let $\mathbb{R}^3 \setminus  \overline{B}$ be connected. Later, we will use the result of this section by regarding $B$ as auxiliary domains, like $B_1$, $B_2$, and $B_3$ in Theorems 1.2 and 1.4. We define $G^{Dir}_{B}:H^{1/2}(\partial B) \to L^{2}(\mathbb{S}^{2})$ by
\begin{equation}
G^{Dir}_{B}f:=v^{\infty},
\label{2.13}
\end{equation}
where $v^{\infty}$ is the far field pattern of a radiating solution $v$ (that is, $v$ satisfies the Sommerfeld radiation condition) such that   
\begin{equation}
\Delta v+k^2v=0 \ \mathrm{in} \ \mathbb{R}^3\setminus  \overline{B}, \label{2.14}
\end{equation}
\begin{equation}
v=f \ \mathrm{on} \ \partial{B}. \label{2.15}
\end{equation}
Let $\lambda_0>0$. We also define $G^{Imp}_{B,i\lambda_{0}}:H^{-1/2}(\partial B) \to L^{2}(\mathbb{S}^{2})$ in the same way as $G^{Dir}_{B}$ by replacing (\ref{2.15}) with
\begin{equation}
\frac{\partial v}{\partial \nu_{B}}+i\lambda_{0}v=f \ \mathrm{on} \ \partial{B}. \label{2.16}
\end{equation}
We define the boundary integral operators $S_{B}:H^{-1/2}(\partial B) \to H^{1/2}(\partial B)$ and $N_{B}:H^{1/2}(\partial B) \to H^{-1/2}(\partial B)$ by
\begin{equation}
S_{B}\varphi(x):=\int_{\partial B} \varphi(y)\Phi(x,y)ds(y), \ x \in \partial B, \label{2.4}
\end{equation}
\begin{equation}
N_{B}\psi(x):=\frac{\partial}{\partial \nu_{B}(x)}\int_{\partial B} \psi(y)\frac{\partial\Phi(x,y)}{\partial \nu_{B}(y)}ds(y), \ x \in \partial B, \label{2.5}
\end{equation}
where
$\Phi(x,y):= \displaystyle \frac{\mathrm{e}^{ik|x-y|}}{4 \pi |x-y|}$. We also define $S_{B,i}$ and $N_{B,i}$ by the boundary integral operators ($\ref{2.4}$) and ($\ref{2.5}$), respectively, corresponding to the wave number $k=i$. It is well known that $S_{B,i}$ is self-adjoint and  positive coercive, and $N_{B,i}$ is self-adjoint and negative coercive. For details of the boundary integral operators, we refer to \cite{Kirsch and Grinberg} and \cite{McLean}.
\par
The following properties of far field operators $F^{Dir}_{B}$ and $F^{Imp}_{B,i\lambda_0}$ are given by previous works in \cite{Kirsch and Grinberg} and \cite{Kirsch and Liu}:

\begin{lem}[Lemma 1.14 in \cite{Kirsch and Grinberg}, Theorem 2.1 and Lemma 2.2 in \cite{Kirsch and Liu}]\ \ \ 
\vspace{-0.2cm}
\begin{description}
\item[(a)] The far field operators $F^{Dir}_{B}$ and $F^{Imp}_{B,i\lambda_0}$ have a factorization of the form 
\begin{equation}
F^{Dir}_{B}=-G^{Dir}_{B}S^{*}_{B}G^{Dir\ *}_{B},\ \ \ \ \ F^{Imp}_{B,i\lambda_{0}}=-G^{Imp}_{B,i\lambda_{0}}T^{Imp\ *}_{B,i\lambda_0}G^{Imp\ *}_{B,i\lambda_{0}}. \label{2.17}
\end{equation}

\item[(b)] The operators $S_{B}:H^{-1/2}(\partial B) \to H^{1/2}(\partial B)$ and $T^{Imp}_{B,i\lambda_0}:H^{1/2}(\partial B) \to H^{-1/2}(\partial B)$ is of the form
\begin{equation}
S_{B}=S_{B,i}+K,\ \ \ \ T^{Imp}_{B,i\lambda_0}=N_{B,i}+K',\label{2.18}
\end{equation}
where $K$ and $K'$ are some compact operators.

\item[(c)] 
$\mathrm{Im} \langle \varphi, S_{B} \varphi\rangle \leq 0$ \ for \ all \ $\varphi \in H^{-1/2}(\partial B)$. Furthermore, if we assume that $k^2$ is not a Dirichlet eigenvalue of $-\Delta$ in $B$, then $\mathrm{Im} \langle \varphi, S_{B} \varphi\rangle < 0$ \ for \ all \ $\varphi \in H^{-1/2}(\partial B)$ with $\varphi \neq 0$.

\item[(d)] 
$\mathrm{Im} \langle T^{Imp}_{B,i\lambda_{0}} \varphi, \varphi\rangle > 0$ \ for \ all \ $\varphi \in H^{1/2}(\partial B)$ with $\varphi \neq 0$.

\end{description}
\end{lem}

Secondly, we consider the far field operator $F^{Mix}_{\Omega_1,\Omega_2}$ for the problem (\ref{1.2})--(\ref{1.5}). Recall that $\Omega=\Omega_1 \cup \Omega_2$, and $\Omega_1$ is an impenetrable obstacle with Dirichlet boundary condition, and $\Omega_2$ with Neumann boundary condition. We define $G^{Mix}_{\Omega_1,\Omega_2}:H^{1/2}(\partial \Omega_1) \times H^{-1/2}(\partial \Omega_2) \to L^{2}(\mathbb{S}^{2})$ by
\begin{equation}
G^{Mix}_{\Omega_1,\Omega_2}
\left(
    \begin{array}{cc}
      f \\
      g 
    \end{array}
  \right)
  :=v^{\infty},
\label{2.1}
\end{equation}
where $v^{\infty}$ is the far field pattern of a radiating solution $v$ such that   
\begin{equation}
\Delta v+k^2v=0 \ \mathrm{in} \ \mathbb{R}^3\setminus  \overline{\Omega}, \label{2.2}
\end{equation}
\begin{equation}
v=f \ \mathrm{on} \ \partial{\Omega_1}, \ \ \ \ \ \frac{\partial v}{\partial \nu_{\Omega_2}}=g \ \mathrm{on} \ \partial{\Omega_2}. \label{2.3}
\end{equation}
The following properties of $F^{Mix}_{\Omega_1,\Omega_2}$ are given by previous works in \cite{Kirsch and Grinberg}:

\begin{lem}[Theorem 3.4 in \cite{Kirsch and Grinberg}]
\begin{description}

\item[(a)] The far field operator $F^{Mix}_{\Omega_1,\Omega_2}$ has a factorization of the form 
\begin{equation}
F^{Mix}_{\Omega_1,\Omega_2}=-G^{Mix}_{\Omega_1,\Omega_2}T^{Mix\ *}_{\Omega_1,\Omega_2}G^{Mix\ *}_{\Omega_1,\Omega_2}. \label{2.6}
\end{equation}
\item[(b)] The middle operator $T^{Mix}_{\Omega_1,\Omega_2}: H^{-1/2}(\partial \Omega_1) \times H^{1/2}(\partial \Omega_2) \to H^{1/2}(\partial \Omega_1) \times H^{-1/2}(\partial \Omega_2)$ is of the form
\begin{equation}
T^{Mix}_{\Omega_1,\Omega_2}=
\left(
    \begin{array}{cc}
      S_{\Omega_1,i} & 0 \\
      0 & N_{\Omega_2,i}
    \end{array}
  \right)
+K, \label{2.7}
\end{equation}
where $K$ is some compact operator.

\item[(c)] 
$\mathrm{Im} \langle T^{Mix}_{\Omega_1,\Omega_2} \varphi, \varphi\rangle \geq 0$ \ for \ all \ $\varphi \in H^{-1/2}(\partial \Omega_1) \times H^{1/2}(\partial \Omega_2)$.
\end{description}
\end{lem}

Thirdly, we consider the far field operator $F^{Mix}_{\Omega_1q,\Omega_2}$ for the problem (\ref{1.6})--(\ref{1.8}). Here, $\Omega_1$ is a penetrable medium modeled by a contrast function $q\in L^{\infty}(\Omega_1)$, and $\Omega_2$ is an impenetrable obstacle with Dirichlet boundary condition. We define $G^{Mix}_{\Omega_1q,\Omega_2}: L^{2}( \Omega_1) \times H^{1/2}(\partial \Omega_2) \to L^{2}(\mathbb{S}^{2})$ by
\begin{equation}
G^{Mix}_{\Omega_1q,\Omega_2}
\left(
    \begin{array}{cc}
      f \\
      g 
    \end{array}
  \right)
  :=v^{\infty},
\label{2.8}
\end{equation}
where $v^{\infty}$ is the far field pattern of a radiating solution $v$ such that   
\begin{equation}
\Delta v+k^2(1+q)v=-k^2 \frac{q}{\sqrt{|q|}}f \ \mathrm{in} \ \mathbb{R}^3\setminus  \overline{\Omega_2}, \label{2.9}
\end{equation}
\begin{equation}
v=-g \ \mathrm{on} \ \partial{\Omega_2}. \label{2.11}
\end{equation}
The following properties of $F^{Mix}_{\Omega_1q,\Omega_2}$ are given by previous works in \cite{Kirsch and Liu2}: 

\begin{lem}[Theorem 3.2 and Theorem 3.3 in \cite{Kirsch and Liu2}]
\begin{description}
\item[(a)] The far field operator $F^{Mix}_{\Omega_1q,\Omega_2}$ has a factorization of the form 
\begin{equation}
F^{Mix}_{\Omega_1q,\Omega_2}=G^{Mix}_{\Omega_1q,\Omega_2}M^{Mix\ *}_{\Omega_1q,\Omega_2}G^{Mix\ *}_{\Omega_1q,\Omega_2}. \label{2.11}
\end{equation}
\item[(b)] The middle operator $M^{Mix}_{\Omega_1q,\Omega_2}: L^{2}(\Omega_1) \times H^{-1/2}(\partial \Omega_2) \to L^{2}(\Omega_1) \times H^{1/2}(\partial \Omega_2)$ is of the form
\begin{equation}
M^{Mix}_{\Omega_1q,\Omega_2}=
\left(
    \begin{array}{cc}
      \frac{|q|}{k^2q} & 0 \\
      0 & -S_{\Omega_2,i}
    \end{array}
  \right)
+K, \label{2.12}
\end{equation}
where $K$ is some compact operator.

\item[(c)] 
$\mathrm{Im} \langle \varphi,  M^{Mix}_{\Omega_1q,\Omega_2}\varphi \rangle \geq 0$ \ for \ all \ $\varphi \in L^{2}(\Omega_1) \times H^{-1/2}(\partial \Omega_2)$.

\item[(d)] 
If \ $M^{Mix}_{\Omega_1q,\Omega_2}\varphi=0$
    ,\ $\varphi=\left( \begin{array}{cc}
      \varphi_1 \\
      \varphi_2
    \end{array}\right) \in L^{2}(\Omega_1) \times H^{-1/2}(\partial \Omega_2)$, then $\varphi_1=0$.
\end{description}
\end{lem}

Finally, we give the following functional analytic theorem behind the factorization method. The proof is completely analogous to previous works, e.g., Theorem 2.15 in \cite{Kirsch and Grinberg},  Theorem 2.1 in \cite{Lechleiter}, and Theorem 2.1 in \cite{Liu}.

\begin{thm}
Let $X \subset U\subset X^{*}$ be a Gelfand triple with a Hilbert space $U$ and a reflexive Banach space $X$ such that the imbedding is dense. Furthermore, let Y be a second Hilbert space and let $F:Y \to Y$, $G:X \to Y$, $T:X^{*} \to X$ be linear bounded operators such that 
\begin{equation}
F=GTG^{*}.\label{2.19}
\end{equation}
We make the following assumptions:
\begin{description}
\item[(1)]G is compact with dense range in Y.

\item[(2)]There exists $t\in [0,2 \pi]$ such that $\mathrm{Re}(\mathrm{e}^{it}T)$ has the form $\mathrm{Re}(\mathrm{e}^{it}T)=C+K$ with some compact operator $K$ and some self-adjoint and positive coercive operator $C$, i.e., there exists $c>0$ such that
\begin{equation}
\langle \varphi,  C \varphi \rangle \geq c \left\| \varphi \right\|^2 \ for \ all \ \varphi \in X^{*}. \label{2.20}
\end{equation}

\item[(3)]$\mathrm{Im} \langle \varphi, T \varphi \rangle \geq 0$ or $\mathrm{Im} \langle \varphi, T \varphi \rangle \leq 0$ for all $\varphi \in X^{*}$.
\end{description}
Furthermore, we assume that one of the following assumptions:
\begin{description}
\item[(4)]
$T$ is injective.
\item[(5)]
$\mathrm{Im} \langle \varphi, T \varphi \rangle > 0$ or $\mathrm{Im} \langle \varphi, T \varphi \rangle < 0$ for all $\varphi \in \overline{\mathrm{Ran}(G^{*})}$ with $\varphi \neq 0$.
\end{description}
Then, the operator $F_{\#}:=\bigl|\mathrm{Re}(\mathrm{e}^{it}F)\bigr|+\bigl|\mathrm{Im}F\bigr|$ is positive, and the ranges of $G:X \to Y$ and $F_{\#}^{1/2}:Y \to Y$ coincide with each other.
\end{thm}
Remark that, in this paper, the real part and the imaginary part of an operator $A$ are self-adjoint operators given by
\begin{equation}
\mathrm{Re}(A)=\displaystyle \frac{A+A^{*}}{2} \ \ \ \mathrm{and} \ \ \ \mathrm{Im}(A)=\displaystyle \frac{A-A^{*}}{2i}.
\end{equation}

\section{The first case}
In section 3, we prove Theorem 1.2. Let Assumption 1.1 hold. We define $R_1:H^{1/2}(\partial \Omega_1) \times H^{-1/2}(\partial \Omega_2) \to H^{1/2}(\partial \Omega_1) \times H^{1/2}(\partial B_2)$ by 
\begin{equation}
R_1\left(
    \begin{array}{cc}
      f_1 \\
      g_1 
    \end{array}
  \right)
  :=\left(
    \begin{array}{cc}
      f_1 \\
      v_1 \bigl|_{\partial B_2} 
    \end{array}
  \right),\label{3.1}
\end{equation}  
where $v_1$ is a radiating solution such that 

\begin{equation}
\Delta v_1+k^2v_1=0 \ \mathrm{in} \ \mathbb{R}^3\setminus  \overline{\Omega}, \label{3.2}
\end{equation}
\begin{equation}
v_1=f_1 \ \mathrm{on} \ \partial{\Omega_1}, \ \ \ \ \ \frac{\partial v_1}{\partial \nu_{\Omega_2}}=g_1 \ \mathrm{on} \ \partial{\Omega_2}. \label{3.3}
\end{equation}
Then, from the definition of $R_1$, we obtain
\begin{equation}
G^{Mix}_{\Omega_1,\Omega_2}=G^{
Dir}_{\Omega_1, B_2}R_1,\label{3.4}
\end{equation}
where $G^{Dir}_{\Omega_1, B_2}:H^{1/2}(\partial \Omega_1) \times H^{1/2}(\partial B_2) \to L^{2}(\mathbb{S}^{2})$ is also defined for the pure Dirichlet boundary condition on $\Omega_1$ and $B_2$ in the same way as $G^{Mix}_{\Omega_1,\Omega_2}$. (See (\ref{2.1}).)\par
Next, we define $R_2:H^{1/2}(\partial B_2) \to H^{1/2}(\partial \Omega_1) \times H^{1/2}(\partial B_2)$ by

\begin{equation}
  R_2 f_2
  :=\left(
    \begin{array}{cc}
      v_2\bigl|_{\partial \Omega_1}  \\
      f_2
    \end{array}
  \right),\label{3.5}
\end{equation}
where $v_2$ is a radiating solution such that 

\begin{equation}
\Delta v_2+k^2v_2=0 \ \mathrm{in} \ \mathbb{R}^3\setminus  \overline{B_2}, \label{3.6}
\end{equation}
\begin{equation}
v_2=f_2 \ \mathrm{on} \ \partial{B_2}, \label{3.7}
\end{equation}
Then, from the definition of $R_2$, we obtain
\begin{equation}
G^{Dir}_{B_2}=G^{Dir}_{\Omega_1, B_2}R_2. \label{3.8}
\end{equation}
Here, take a positive number $\lambda_0>0$, and a bounded domain $B_3$ with $\overline{B_3} \subset B_2$. We define $R_3:H^{-1/2}(\partial B_1 \cup \partial B_3) \to H^{1/2}(\partial \Omega_1) \times H^{1/2}(\partial B_2)$ by

\begin{equation}
  R_3 f_3
  :=\left(
    \begin{array}{cc}
      v_3\bigl|_{\partial \Omega_1}  \\
      v_3\bigl|_{\partial B_2}
    \end{array}
  \right),\label{3.9}
\end{equation}
where $v_3$ is a radiating solution such that 

\begin{equation}
\Delta v_3+k^2v_3=0 \ \mathrm{in} \ \mathbb{R}^3\setminus  \overline{B_1 \cup B_3}, \label{3.10}
\end{equation}
\begin{equation}
\frac{\partial v_3}{\partial \nu_{B_1\cup B_3}}+i\lambda_0 v_3=f_3 \ \mathrm{on} \ \partial B_1 \cup \partial B_3. \label{3.11}
\end{equation}
Then, from the definition of $R_3$, we obtain
\begin{equation}
G^{Imp}_{B_1\cup B_3,i\lambda_0}=G^{Dir}_{\Omega_1 , B_2}R_3.\label{3.12}
\end{equation}
By (\ref{3.4}), (\ref{3.8}), (\ref{3.12}), and the factorization of the far field operator in Section 2, we have
\begin{equation}
F^{Mix}_{\Omega_1,\Omega_2}+F^{Dir}_{B_2}+F^{Imp}_{B_1 \cup B_3,i\lambda_0}
=G^{Dir}_{\Omega_1,B_2}TG^{Dir\ *}_{\Omega_1,B_2}, \label{3.13}
\end{equation}
where $T:=\Bigl[-R_1T^{Mix\ *}_{\Omega_1,\Omega_2}R^{*}_1-R_2S^{*}_{B_2}R^{*}_2-R_3T^{Imp\ *}_{B_1 \cup B_3,i\lambda_0}R^{*}_3 \Bigr]$.\par
The following properties of $G^{Dir}_{\Omega_1, B_2}$ are given by the same argument in Theorem 1.12 and Lemma 1.13 in \cite{Kirsch and Grinberg}:

\begin{lem}
\begin{description}
\item[(a)]
The operator $G^{Dir}_{\Omega_1, B_2}:H^{1/2}(\partial \Omega_1) \times H^{1/2}(\partial B_2) \to L^{2}(\mathbb{S}^{2})$ is compact with dense range in $L^{2}(\mathbb{S}^{2})$.
\item[(b)]
For $z \in \mathbb{R}^3 \setminus  \overline{B_2}$ 
\begin{equation}
z \in \Omega_1
\Longleftrightarrow
\phi_z \in \mathrm{Ran}(G^{Dir}_{\Omega_1, B_2}), \label{3.14}
\end{equation}
where the function $\phi_z$ is given by $($$\ref{1.11}$$)$.
\end{description}
\end{lem}

To prove Theorem 1.2, we apply Theorem 2.4 to this case. First of all, we show the following lemma:
\begin{lem}
\begin{description}
\item[(a)] 
 $R_1-\left(
    \begin{array}{cc}
      I & 0 \\
      0 & 0
    \end{array}
  \right)$, \ 
  $R_2-P_2$, $R_3$
are  compact. Here, $P_2:H^{1/2}(\partial B_2) \to H^{1/2}(\partial \Omega_1) \times H^{1/2}(\partial B_2)$
is defined by
\begin{equation}
  P_2 h
  :=\left(
    \begin{array}{cc}
      0 \\
      h
    \end{array}
  \right). \label{3.15}
\end{equation}

\item[(b)]
$R_3^{*}$ is injective.
\end{description}
\end{lem}
\begin{proof}[Proof]
(a) The mappings 
$R_1-\left(
    \begin{array}{cc}
      I & 0 \\
      0 & 0
    \end{array}
  \right)$ : $H^{1/2}(\partial \Omega_1) \times H^{-1/2}(\partial \Omega_2) \to H^{1}(\partial \Omega_1) \times H^{1}(\partial B_2)$, $R_2-P_2$ : $H^{1/2}(\partial B_2) \to H^{1}(\partial \Omega_1) \times H^{1}(\partial B_2)$, and $R_3$ : $H^{-1/2}(\partial B_1 \cup \partial B_3) \to H^{1}(\partial \Omega_1) \times H^{1}(\partial B_2)$ are bounded since they are given by $\left(
     \begin{array}{cc}
      f_1  \\
      g_1 
    \end{array}
    \right)
    \mapsto
    \left(
     \begin{array}{cc}
      0  \\
      v_1\bigl|_{\partial B_2}
    \end{array}
    \right)$, 
    $f_2
  \mapsto \left(
    \begin{array}{cc}
      v_2\bigl|_{\partial \Omega_1}  \\
      0
    \end{array}
  \right)$, and $f_3 \mapsto \left(
    \begin{array}{cc}
      v_3\bigl|_{\partial \Omega_1}  \\
      v_3\bigl|_{\partial B_2}
    \end{array}
  \right)$,
     respectively. By Rellich theorem, they are compact.\par

(b) Let $\phi \in H^{-1/2}(\partial \Omega_1)$ and $\psi \in H^{-1/2}(\partial B_2)$. Assume that $R_3^{*}\left(
    \begin{array}{cc}
      \phi \\
      \psi
    \end{array}
  \right)=0$.
Using the same argument as done in Theorem 2.5 in \cite{Liu and Zhang and Hu}, one knows the existence of a radiating solution $w$ such that
\begin{equation}
\Delta w+k^2 w=0 \ \mathrm{in} \ \mathbb{R}^3\setminus  \overline{\Omega_1 \cup B_2}, \label{3.16}
\end{equation}
\begin{equation}
\Delta w+k^2w=0 \ \mathrm{in} \ \Omega_1 \setminus \overline{B_1}, \ \mathrm{in} \ B_2 \setminus \overline{B_3},\label{3.17}
\end{equation}

\begin{equation}
w_+-w_-=0, \ \ \ \ \  \frac{\partial w_+}{\partial \nu_{\Omega_1}}-\frac{\partial w_-}{\partial \nu_{\Omega_1}}=\overline{\phi} \ \ \mathrm{on} \ \partial{\Omega_1},  \label{3.18}
\end{equation}
\begin{equation}
w_+-w_-=0, \ \ \ \ \  \frac{\partial w_+}{\partial \nu_{B_2}}-\frac{\partial w_-}{\partial \nu_{B_2}}=\overline{\psi} \ \ \mathrm{on} \ \partial{B_2},  \label{3.19}
\end{equation}

\begin{equation}
\frac{\partial w}{\partial \nu_{B_1}}+i\lambda_0 w=0 \ \mathrm{on} \ \partial{B_1}, \ \ \ \ \ \ \frac{\partial w}{\partial \nu_{B_3}}+i\lambda_0 w=0 \ \mathrm{on} \ \partial{B_3}, \label{3.20}
\end{equation}
where the subscripts + and -- denote the trace from the exterior and interior, respectively. (See Figure 3).

\begin{figure}[h]
\centering
\begin{tikzpicture}[scale=0.8]
\draw [very thick](6,2) circle [x radius=2cm, y radius=1.4cm, rotate=0];
\draw [very thick](1.5,2) circle [x radius=1.4cm, y radius=10mm, rotate=100];
\node (A) at (1.6,1.2) [above] {{\large $\Omega_1$}} ;
\node (C) at (5.3,1) [above] {{\large $B_2$}} ;
\fill [black!20](1.5,2.7) circle (0.5);
\node (E) at (1.5,2.3) [above] {{\large $B_1$}} ;
\fill [black!20](6.5,2.5) circle (0.5);
\node (E) at (6.55,2.1) [above] {{\large $B_3$}} ;
\end{tikzpicture}
\caption{}
\end{figure}

By using the boundary conditions (\ref{3.11}), (\ref{3.18}), (\ref{3.19}), (\ref{3.20}), and Green's theorem, we have
\begin{eqnarray}
0&=&\Bigl \langle f_3 , R_3^{*}\left(
    \begin{array}{cc}
      \phi \\
      \psi
    \end{array}
  \right)\Bigr \rangle
= \Bigl \langle \left(
    \begin{array}{cc}
      v_3\bigl|_{\partial \Omega_1}  \\
      v_3\bigl|_{\partial B_2}
    \end{array}
  \right) , \left(
    \begin{array}{cc}
      \phi \\
      \psi
    \end{array}
  \right)\Bigr \rangle
\nonumber\\
&=& 
\int_{\partial \Omega_1} v_3 \overline{\phi} ds+\int_{\partial B_2} v_3 \overline{\psi} ds
\nonumber\\
&=&
\int_{\partial \Omega_1 \cup \partial B_2} v_3 \Bigl( \frac{\partial w_+}{\partial \nu}-\frac{\partial w_-}{\partial \nu} \Bigr) ds - \int_{\partial \Omega_1 \cup \partial B_2} \frac{\partial v_3}{\partial \nu} ( w_+ - w_- ) ds
\nonumber\\
&=&
\int_{\partial \Omega_1 \cup \partial B_2} \biggl[ \frac{\partial v_3}{\partial \nu} w_-  - v_3 \frac{\partial w_-}{\partial \nu} \biggr] ds - \int_{\partial \Omega_1 \cup \partial B_2} \biggl[  \frac{\partial v_3}{\partial \nu} w_+  - v_3 \frac{\partial w_+}{\partial \nu} \biggr] ds
\nonumber\\
&=&
\int_{\partial B_1} \biggl[ \frac{\partial v_3}{\partial \nu_{B_1}} w - v_3 \frac{\partial w}{\partial \nu_{B_1}} \biggr] ds + \int_{\partial B_3} \biggl[  \frac{\partial v_3}{\partial \nu_{B_3}} w - v_3 \frac{\partial w}{\partial \nu_{B_3}} \biggr] ds  \nonumber\\
&=& \int_{\partial B_1 \cup \partial B_3} f_3 wds,\label{3.21}
\end{eqnarray}
which proves that $w=0 \ \mathrm{in} \ \partial B_1 \cup \partial B_3$. Holmgren's uniqueness theorem (See e.g., Theorem 2.3 in \cite{Colton and Kress}) implies that $w$ vanishes in $\Omega_1 \setminus \overline{B_1}$ and $B_2 \setminus \overline{B_3}$. Equations ($\ref{3.18}$) and ($\ref{3.19}$) yield $w_+=0 \ \mathrm{on} \ \partial{\Omega_1} \cup \partial{B_2}$ which implies that $w$ vanishes also outside of $\Omega_1$ and $B_2$ by the uniqueness of the exterior Dirichlet problem. Therefore, equations ($\ref{3.18}$) and ($\ref{3.19}$) yield $\phi=0$ and $\psi=0$.
\end{proof}
By Lemma 3.2, the middle operator $T$ of ($\ref{3.13}$) has the following properties:
\begin{lem}
\begin{description}
\item[(a)]
$\mathrm{Re}\bigl(\mathrm{e}^{i\pi}T \bigr)$ has the form $\mathrm{Re}\bigl(\mathrm{e}^{i\pi}T \bigr)=C+K$ with some self-adjoint and positive coercive operator $C$ and some compact operator $K$.
\item[(b)]$
\mathrm{Im} \langle \varphi, T \varphi\rangle<0 \ for \ all \ \varphi \in H^{-1/2}(\partial \Omega_1) \times H^{-1/2}(\partial B_2) \ with \ \varphi \neq 0
$.
\end{description}
\end{lem}
\begin{proof}[Proof]
(a) By Lemma 2.1 (b), Lemma 2.2 (b), and Lemma 3.2 (a),
\begin{eqnarray}
\mathrm{Re}\bigl(\mathrm{e}^{i\pi}T \bigr)
&=&
\mathrm{Re}\Bigl(R_1T^{Mix\ *}_{\Omega_1,\Omega_2}R^{*}_1+R_2S^{*}_{B_2}R^{*}_2+R_3T^{Imp\ *}_{B_1 \cup B_3,i\lambda_0}R^{*}_3  \Bigr)
\nonumber\\
&=&
R_1 
\left(
    \begin{array}{cc}
      S_{\Omega_1,i} & 0 \\
      0 & N_{\Omega_2,i}
    \end{array}
\right)
R^{*}_1+R_2S_{B_2,i} R^{*}_2+K
\nonumber\\
&=&
\left(
    \begin{array}{cc}
      I & 0 \\
      0 & 0
    \end{array}
\right)
\left(
    \begin{array}{cc}
      S_{\Omega_1,i} & 0 \\
      0 & N_{\Omega_2,i}
    \end{array}
\right)
\left(
    \begin{array}{cc}
      I & 0 \\
      0 & 0
    \end{array}
\right)
+
P_2 S_{B_2,i}P^{*}_2
+K'
\nonumber\\
&=&
\left(
    \begin{array}{cc}
      S_{\Omega_1,i} & 0 \\
      0 & S_{B_2,i}
    \end{array}
\right)+K',\label{3.22}
\end{eqnarray}
where $K$ and $K'$ are some compact operators. Since the boundary integral operators $S_{\Omega_1, i}$ and $S_{B_2, i}$ are self-adjoint and positive coercive, (a) holds.
\par
(b) By Lemma 2.1 (c) (d), Lemma 2.2 (c), and Lemma 3.2 (b), especially, by the strictly positivity of the operator $\mathrm{Im}T^{Imp}_{B_1 \cup B_3,i\lambda_0}$, and the injectivity of $R^{*}_3$, for all $\varphi \in H^{-1/2}(\partial \Omega_1) \times H^{-1/2}(\partial B_2)$ with $\varphi \neq 0$, we have
\begin{eqnarray}
\mathrm{Im} \langle \varphi, T \varphi\rangle
&=&-\mathrm{Im} \langle T^{Mix}_{\Omega_1, \Omega_2} R^{*}_1 \varphi,R^{*}_1 \varphi \rangle
+\mathrm{Im}\langle R^{*}_2 \varphi, S_{B_2}R^{*}_2 \varphi \rangle
\nonumber\\
&&
-\mathrm{Im}\langle T^{Imp}_{B_1 \cup B_3,i\lambda_0} R^{*}_3 \varphi,R^{*}_3 \varphi \rangle
<0. \label{3.23}
\end{eqnarray}
\end{proof}
Therefore, by Lemma 3.3, we can apply Theorem 2.4 to this case. From Lemma 3.1 (b), and applying Theorem 2.4, we obtain Theorem 1.2.
\begin{rem}
Unknown obstacle $\Omega_2$ may consist of finitely many connected components whose closures are mutually disjoint. Furthermore, the boundary condition on $\Omega_2$ can not be only Neumann but also Dirichlet, impedance, and not only impenetrable obstacles but also penetrable mediums, and their mixed situations by the same argument in Theorem 1.2. In all cases, we can choose arbitrary wave numbers $k>0$.
\end{rem}

\begin{rem}
If we assume that $k^2$ is not a Dirichlet eigenvalue of $-\Delta$ in artificial domains $B_1$, $B_2$, then we do not need to take an additional domain $B_3$. In such a case, we only use $F^{Dir}_{B_1 \cup B_2}$ as artificial far field operators since $F^{Dir}_{B_1 \cup B_2}$ has a role to keep the strictly positivity of the imaginary part of the middle operator of $F$. (See Lemma 2.1 (c).) That is, we can give the following characterization by the same argument in Theorem 1.2:
\end{rem}

\begin{thm}\label{thm1.2} 
In addition to $\mathrm{Assumption \ 1.1}$, we assume that $k^2$ is not a Dirichlet eigenvalue of $-\Delta$ in $B_1$, $B_2$. Take a positive number $\lambda_0>0$. Then, for $z \in \mathbb{R}^3 \setminus  \overline{B_2}$
\begin{equation}
z \in \Omega_1
\Longleftrightarrow
\sum_{n=1}^{\infty}\frac{|(\phi_z,\varphi_n)_{L^2(\mathbb{S}^{2})}|^2}{\lambda_n} < \infty, \label{1.12}
\end{equation}
where $(\lambda_n,\varphi_n)$ is a complete eigensystem of $F_{\#}$ given by
\begin{equation}
F_{\#}:=\bigl|\mathrm{Re}F\bigr|+\bigl|\mathrm{Im}F\bigr|,\label{1.13}
\end{equation}
where $F:=F^{Mix}_{\Omega_1,\Omega_2}+F^{Dir}_{B_1 \cup B_2}$. Here, the function $\phi_z$ is given by $($$\ref{1.11}$$)$.
\end{thm}

\begin{rem}
We can also give the characterization of the Neumann part $\Omega_2$ if we assume $\overline{\Omega_1}\subset B_1$, $\overline{B_2}\subset \Omega_2$, $\overline{B_1} \cap \overline{\Omega_2}=\emptyset$ by the same argument in Theorem 1.2 $($See $\mathrm{Figure\ 4}$$)$.
\end{rem}

\begin{figure}[h]
\centering
\begin{tikzpicture}[scale=0.7]
\fill [black!20](6,2) circle [x radius=1.4cm, y radius=1.1cm, rotate=130];
\fill [black!20](1.5,2) circle [x radius=1.4cm, y radius=10mm, rotate=100];
\node (A) at (1.5,1.7) [above] {{\large$\Omega_1$}} ;
\node (B) at (6,1.2) [above] {{\large$\Omega_2$}} ;
\node (C) at (5.7,2.3) [above] {{\large $B_2$}} ;
\node (H) at (5.9,0.8) [above] {\large Neumann} ;
\node (H) at (1.5,1.2) [above] {\large Dirichlet} ;
\draw [very thick](1.5,2) circle  [x radius=2cm, y radius=15mm, rotate=50];;
\draw [very thick](5.7,2.7) circle [x radius=0.5cm, y radius=0.4cm, rotate=0];
\node (E) at (0,2.9) [above] {{\large $B_1$}} ;
\end{tikzpicture}
\caption{}
\end{figure}



\section{The second case}
In Section 4, we prove Theorem 1.4. Let Assumption 1.3 hold. We define $G^{Mix}_{\Omega_10,B_2}: L^{2}(\Omega_1) \times H^{1/2}(\partial B_2) \to L^{2}(\mathbb{S}^{2})$ by
\begin{equation}
G^{Mix}_{\Omega_10,B_2}
\left(
    \begin{array}{cc}
      f \\
      g 
    \end{array}
  \right)
  :=v^{\infty},
\label{4.1}
\end{equation}
where $v^{\infty}$ is the far field pattern of a radiating solution $v$ such that   
\begin{equation}
\Delta v+k^2v=-k^2 \frac{q}{\sqrt{|q|}}f \ \mathrm{in} \ \mathbb{R}^3\setminus  \overline{B_2}, \label{4.2}
\end{equation}
\begin{equation}
v=g \ \mathrm{on} \ \partial{B_2}. \label{4.3}
\end{equation}
Note that we extend $q$ by zero outside $\Omega_1$.
Next, we define $R_1:L^{2}(\Omega_1) \times H^{1/2}(\partial \Omega_2) \to L^{2}(\Omega_1) \times H^{1/2}(\partial B_2)$ by
\begin{equation}
R_1
\left(
    \begin{array}{cc}
      f_1 \\
      g_1 
    \end{array}
  \right)
  :=\left(
    \begin{array}{cc}
      f_1+\sqrt{|q|}v_1 \\
      v_1\bigl|_{\partial B_2}
    \end{array}
  \right),
\label{4.4}
\end{equation}
where $v_1$ is a radiating solution such that   
\begin{equation}
\Delta v_1+k^2(1+q)v_1=-k^2 \frac{q}{\sqrt{|q|}}f_1 \ \mathrm{in} \ \mathbb{R}^3\setminus  \overline{\Omega_2}, \label{4.5}
\end{equation}
\begin{equation}
v_1=-g_1 \ \mathrm{on} \ \partial{\Omega_2}. \label{4.6}
\end{equation}
Then, from the definition of $R_1$, we obtain
\begin{equation}
G^{Mix}_{\Omega_1q,\Omega_2}=G^{Mix}_{\Omega_10,B_2}R_1.\label{4.7}
\end{equation}
We define $R_2:H^{1/2}(\partial B_2) \to L^{2}(\Omega_1) \times H^{1/2}(\partial B_2)$ by
\begin{equation}
R_2
f_2:=\left(
    \begin{array}{cc}
      0 \\
      f_2
    \end{array}
  \right).
\label{4.8}
\end{equation}
Then, from the definition of $R_2$, we obtain
\begin{equation}
G^{Dir}_{B_2}=G^{Mix}_{\Omega_10,B_2}R_2.\label{4.9}
\end{equation}
Here, take a positive number  $\lambda_0>0$, and a bounded domain $B_3$ with $\overline{B_3} \subset B_2$. We define $R_3:H^{-1/2}(\partial B_3) \to H^{1/2}(\partial B_2)$ by
\begin{equation}
  R_3 f_3:=v_3\bigl|_{\partial B_2},\label{4.10}
\end{equation}
where $v_3$ is a radiating solution such that 
\begin{equation}
\Delta v_3+k^2v_3=0 \ \mathrm{in} \ \mathbb{R}^3\setminus  \overline{B_3}, \label{4.11}
\end{equation}
\begin{equation}
\frac{\partial v_3}{\partial \nu_{B_3}}+i\lambda_0 v_3=f_3 \ \mathrm{on} \ \partial B_3. \label{4.12}
\end{equation}
Then, from the definition of $R_3$, and (\ref{4.9}), we obtain
\begin{equation}
G^{Imp}_{B_3,i\lambda_0}=G^{Dir}_{B_2}R_3=G^{Mix}_{\Omega_10,B_2}R_2R_3.\label{4.13}
\end{equation}
By (\ref{4.7}), (\ref{4.9}), (\ref{4.13}), and the factorization of the far field operator in Section 2, we have
\begin{equation}
F^{Mix}_{\Omega_1q,\Omega_2}+F^{Dir}_{B_2}+F^{Imp}_{B_3,i\lambda_0}
=G^{Mix}_{\Omega_10,B_2}MG^{Mix\ *}_{\Omega_10,B_2}, \label{4.14}
\end{equation}
where $M:=\Bigl[R_1M^{Mix\ *}_{\Omega_1q,\Omega_2}R^{*}_1-R_2S^{*}_{B_2}R^{*}_2-R_2R_3T^{Imp\ *}_{B_3,i\lambda_0}R^{*}_3R^{*}_2 \Bigr]$.\par
The following properties are given by the same argument in Theorem 3.2 (c) in \cite{Kirsch and Liu2}:

\begin{lem}
\begin{description}
\item[(a)]
The operator $G^{Mix}_{\Omega_10,B_2}:L^{2}(\Omega_1) \times H^{1/2}(\partial B_2) \to L^{2}(\mathbb{S}^{2})$ is compact with dense range in $L^{2}(\mathbb{S}^{2})$.
\item[(b)]
For $z \in \mathbb{R}^3 \setminus  \overline{B_2}$ 
\begin{equation}
z \in \Omega_1
\Longleftrightarrow
\phi_z \in \mathrm{Ran}(G^{Mix}_{\Omega_10,B_2}), \label{4.15}
\end{equation}
where the function $\phi_z$ is given by $($$\ref{1.11}$$)$.
\end{description}
\end{lem}

To prove Theorem 1.4, we apply Theorem 2.4 to this case with $F=F^{Mix\ *}_{\Omega_1q,\Omega_2}+F^{Dir\ *}_{B_2}+F^{Imp\ *}_{B_3,i\lambda_0}$. First, we show the following lemma:
\begin{lem}
\begin{description}
\item[(a)] 
 $R_1-\left(
    \begin{array}{cc}
      I & 0 \\
      0 & 0
    \end{array}
  \right)$, $R_3$
are  compact.
\item[(b)] 
$R_1$ is injective.
\item[(c)]
$R_3^{*}$ is injective.
\end{description}
\end{lem}
\begin{proof}[Proof]
(a) The mappings 
$R_1-\left(
    \begin{array}{cc}
      I & 0 \\
      0 & 0
    \end{array}
  \right)$ : $L^{2}(\Omega_1) \times H^{1/2}(\partial \Omega_2) \to H^{1}(\Omega_1) \times H^{1}(\partial B_2)$, and 
$R_3$ : $H^{-1/2}(\partial B_3) \to H^{1}(\partial B_2)$ are bounded since they are given by
$\left(
     \begin{array}{cc}
      f_1  \\
      g_1 
    \end{array}
    \right)
    \mapsto
    \left(
     \begin{array}{cc}
      \sqrt{|q|}v_1  \\
      v_1\bigl|_{\partial B_2}
    \end{array}
    \right)$, and $f_3 \mapsto v_3\bigl|_{\partial B_2}$, respectively. By Rellich theorem, they are compact.\par

(b) Assume that
\begin{equation}
R_1
\left(
    \begin{array}{cc}
      f_1 \\
      g_1 
    \end{array}
  \right)
  =\left(
    \begin{array}{cc}
      f_1+\sqrt{|q|}v_1 \\
      v_1\bigl|_{\partial B_2}
    \end{array}
  \right)=0.
\label{4.16}
\end{equation}
Equation (\ref{4.5}) yields that
\begin{equation}
\Delta v_1+k^2v_1=0 \ \mathrm{in} \ \mathbb{R}^3\setminus  \overline{B_2}, \label{4.17}
\end{equation}
\begin{equation}
v_1=0 \ \mathrm{on} \ \partial{B_2} \label{4.18}.
\end{equation}
By the uniqueness of the exterior Dirichlet problem, $v_1$ vanishes outside of $B_2$. Therefore, $f_1=0$. Furthermore, the analyticity of $v_1$ yields that $v_1$ also vanishes in $B_2 \setminus  \overline{\Omega_2}$, which implies that $g_1=0$.

(c) The injectivity of $R^{*}_3$ follows from the same argument as done in the proof of Lemma 3.2 in \cite{Kirsch and Liu}.
\end{proof}

By Lemma 4.2, the middle operator $M$ of ($\ref{4.14}$) has the following properties:
\begin{lem}
\begin{description}
\item[(a)]
$\mathrm{Re}\bigl(\mathrm{e}^{it}M^{*} \bigr)$ has the form $\mathrm{Re}\bigl(\mathrm{e}^{it}M^{*} \bigr)=C+K$ with some self-adjoint and positive coercive operator $C$, and some compact operator $K$.
\item[(b)]$
\mathrm{Im} \langle \varphi, M^{*} \varphi\rangle \geq 0 \ for \ all \ \varphi \in L^{2}(\Omega_1) \times H^{-1/2}(\partial B_2)$.
\item[(c)] $M^{*}$ is injective.
\end{description}
\end{lem}
\begin{proof}[Proof]

(a) By Lemma 2.1 (b), Lemma 2.3 (b), and Lemma 4.2 (a),
\begin{eqnarray}
\mathrm{Re}\bigl(\mathrm{e}^{it}M^{*} \bigr)
&=&
\mathrm{Re}\Bigl(\mathrm{e}^{it} R_1M^{Mix}_{\Omega_1q,\Omega_2}R^{*}_1-\mathrm{e}^{it}R_2S_{B_2}R^{*}_2-\mathrm{e}^{it}R_2R_3T^{Imp}_{B_3,i\lambda_0}R^{*}_3R^{*}_2 \Bigr)
\nonumber\\
&=&
R_1 
\left(
    \begin{array}{cc}
      \mathrm{Re}(\frac{\mathrm{e}^{it}|q|}{k^2q}) & 0 \\
      0 & -$(cos t)$S_{\Omega_2,i}
    \end{array}
  \right)
R^{*}_1-R_2(\mathrm{cos\ t})S_{B_2,i} R^{*}_2+K
\nonumber\\
&=&
\left(
    \begin{array}{cc}
      I & 0 \\
      0 & 0
    \end{array}
\right)
\left(
    \begin{array}{cc}
      \mathrm{Re}(\frac{\mathrm{e}^{it}|q|}{k^2q}) & 0 \\
      0 & -$(cos t)$S_{\Omega_2,i}
    \end{array}
\right)
\left(
    \begin{array}{cc}
      I & 0 \\
      0 & 0
    \end{array}
\right)
\nonumber\\
&&
\nonumber\\
&&
-R_2(\mathrm{cos\ t})S_{B_2,i} R^{*}_2
+K'
\nonumber\\
&&
\nonumber\\
&=&
\left(
    \begin{array}{cc}
      \mathrm{Re}(\frac{\mathrm{e}^{it}|q|}{k^2q}) & 0 \\
      0 & (-\mathrm{cos\ t})S_{B_2,i} 
    \end{array}
\right)+K',\label{4.19}
\end{eqnarray}
where $K$ and $K'$ are some compact operators. The first term of the right hand side in (\ref{4.19}) is self-adjoint and positive coercive since $(-\mathrm{cos\ t}) >0$ when $t \in (\pi/2,3\pi/2)$, and Assumption 1.3 (iv) yields

\begin{eqnarray}
\Bigl \langle \varphi, \mathrm{Re}\bigl(\frac{\mathrm{e}^{it}|q|}{k^2q}\bigr) \varphi \Bigr \rangle
&=&
\int_{\Omega_1}|\varphi|^2 \frac{\mathrm{Re}(\mathrm{e}^{-it}q)}{k^2|q|}dx
\nonumber\\
&\geq&
\int_{\Omega_1}|\varphi|^2 \frac{C|q|}{k^2|q|}dx
\nonumber\\
&=&
\frac{C}{k^2}
\left\| \varphi \right\|^2_{L^2(\Omega_1)}.\label{4.20}
\end{eqnarray}

(b) By Lemma 2.1 (c), Lemma 2.3 (c) (d), for all $\varphi \in L^{2}(\Omega_1) \times H^{-1/2}(\partial B_2)$
\begin{eqnarray}
\mathrm{Im} \langle \varphi, M^{*} \varphi\rangle
&=&\mathrm{Im} \langle R^{*}_1 \varphi, M^{Mix}_{\Omega_1q, \Omega_2} R^{*}_1 \varphi \rangle
-\mathrm{Im} \langle R^{*}_2 \varphi, S_{B_2}R^{*}_2 \varphi \rangle
\nonumber\\
&&
+\mathrm{Im} \langle T^{Imp}_{B_3,i\lambda_0} R^{*}_3R^{*}_2 \varphi, R^{*}_3R^{*}_2 \varphi \rangle
\geq 0. \label{4.21}
\end{eqnarray}

(c) Let $\phi \in L^{2}(\Omega_1)$ and $\psi \in H^{-1/2}(\partial B_2)$. Assume that $M^{*} \left(
    \begin{array}{cc}
      \phi \\
      \psi
    \end{array}
  \right)=0$.
Inequality (\ref{4.21}) yields that
\begin{equation}
\mathrm{Im} \Bigl \langle T^{Imp}_{B_3,i\lambda_0} R^{*}_3R^{*}_2 \left(
    \begin{array}{cc}
      \phi \\
      \psi
    \end{array}
  \right), R^{*}_3R^{*}_2 \left(
    \begin{array}{cc}
      \phi \\
      \psi
    \end{array}
  \right)\Bigr \rangle=0,\label{4.22}
\end{equation}
which implies that $R^{*}_{3}R^{*}_{2}
\left(\begin{array}{cc}
      \phi \\
      \psi
\end{array}\right)=0$ from Lemma 2.1 (d). By Lemma 4.2 (c), and the definition of $R_2$, we have $\psi=0$. Therefore,

\begin{equation}
M^{*} \left(
    \begin{array}{cc}
      \phi \\
      0
    \end{array}
  \right)=R_1M^{Mix}_{\Omega_1q,\Omega_2}R^{*}_1
  \left(
    \begin{array}{cc}
      \phi \\
      0
    \end{array}
  \right)=0.\label{4.23}
\end{equation}
From Lemma 4.2 (b) and Lemma 2.3 (d), we obtain 
\begin{equation}
R^{*}_1
\left(
    \begin{array}{cc}
      \phi \\
      0
    \end{array}
  \right)=\left(\begin{array}{cc}
      0 \\
      *
    \end{array}
  \right).\label{4.24}
\end{equation}
Finally, we will show $\phi=0$. Let $f_1 \in L^{2}(\Omega_1)$. Take radiating solutions $v_1$ and $w$ such that 
\begin{equation}
\Delta v_1+k^2(1+q)v_1=-k^2 \frac{q}{\sqrt{|q|}}f_1 \ \mathrm{in} \ \mathbb{R}^3\setminus  \overline{\Omega_2}, \label{4.25}
\end{equation}
\begin{equation}
v_1=0 \ \mathrm{on} \ \partial{\Omega_2}, \label{4.26}
\end{equation}
\begin{equation}
\Delta w+k^2(1+q)w=\sqrt{|q|} \overline{\phi} \ \mathrm{in} \ \mathbb{R}^3\setminus  \overline{\Omega_2}, \label{4.27}
\end{equation}
\begin{equation}
w=0 \ \mathrm{on} \ \partial{\Omega_2}. \label{4.28}
\end{equation}
By (\ref{4.24}),
\begin{eqnarray}
0&=&
\Bigl \langle\left(
    \begin{array}{cc}
      f_1 \\
      0
    \end{array}
  \right) , R_1^{*}\left(
    \begin{array}{cc}
      \phi \\
      0
    \end{array}
  \right)\Bigr \rangle
= \Bigl \langle \left(
    \begin{array}{cc}
      f_1+\sqrt{|q|}v_1 \\
      v_1\bigl|_{\partial B_2}
    \end{array}
  \right) , \left(
    \begin{array}{cc}
      \phi \\
      0
    \end{array}
  \right)\Bigr \rangle
\nonumber\\
&=&
\int_{\Omega_1}f_1 \overline{\phi} dx+\int_{\Omega_1} v_1 \sqrt{|q|} \overline{\phi} dx.\label{4.29}
\end{eqnarray}
By (\ref{4.25}) and (\ref{4.27}),
\begin{eqnarray}
\int_{\Omega_1} v_1 \sqrt{|q|} \overline{\phi} dx
&=&
\int_{\Omega_1} v_1 \bigl(\Delta w+k^2(1+q)w \bigr) dx
\nonumber\\
&&
-\int_{\Omega_1} \Bigl( \Delta v_1+k^2(1+q)v_1+k^2 \frac{q}{\sqrt{|q|}}f_1\Bigr) w dx
\nonumber\\
&=&
-\int_{\Omega_1} k^2 \frac{q}{\sqrt{|q|}}f_1w dx
\nonumber\\
&&
+\int_{\Omega_1} (\Delta w)v_1-w (\Delta v_1)dx.\label{4.30}
\end{eqnarray}
By using Green's theorem, (\ref{4.26}), and (\ref{4.28}),
\begin{eqnarray}
\int_{\Omega_1} (\Delta w)v_1-w (\Delta v_1)dx
&=&
\int_{\mathbb{R}^3\setminus  \overline{\Omega_2}} (\Delta w)v_1-w (\Delta v_1)dx
\nonumber\\
&=&
-\int_{\partial \Omega_2} \biggl[  \frac{\partial w}{\partial \nu_{\Omega_2}} v_1 - w \frac{\partial v}{\partial \nu_{\Omega_2}} \biggr] ds
\nonumber\\
&=&
0.\label{4.31}
\end{eqnarray}
By (\ref{4.29})--(\ref{4.31}),
\begin{equation}
\overline{\phi}=k^2\frac{q}{\sqrt{|q|}}w \ \mathrm{in} \  \Omega_1.\label{4.32}
\end{equation}
From (\ref{4.32}), (\ref{4.27}), and (\ref{4.28}), we obtain
\begin{equation}
\Delta w+k^2w=0 \ \mathrm{in} \ \mathbb{R}^3\setminus  \overline{\Omega_2}, \label{4.33}
\end{equation}
\begin{equation}
w=0 \ \mathrm{on} \ \partial{\Omega_2}, \label{4.34}
\end{equation}
which proves that $w$ vanishes in $\mathbb{R}^3\setminus  \overline{\Omega_2}$ by the uniqueness of the exterior Dirichlet problem. Therefore, equation (\ref{4.32}) yields that 
$\phi=0$.
\end{proof}
Therefore, by Lemma 4.3, we can apply Theorem 2.4 to this case with $F=F^{Mix\ *}_{\Omega_1q,\Omega_2}+F^{Dir\ *}_{B_2}+F^{Imp\ *}_{B_3,i\lambda_0}$. From Lemma 4.1 (b), and applying Theorem 2.4, we obtain Theorem 1.4.

\begin{rem}
We can also consider various situations on $\Omega_2$ like Remark 3.4, and replace the assumption of taking $B_3$ with that $k^2$ is not a Dirichlet eigenvalue of $-\Delta$ in an artificial domain $B_2$ like Remark 3.5.
\end{rem}

We can also give the characterization by replacing (iv) in Assumption 1.3 with
\begin{description}
\item[(iv')] {\it There exists $t \in [0, \pi/2) \cup (3\pi/2, 2 \pi]$ and $C>0$ such that
$\mathrm{Re}(\mathrm{e}^{-it}q)$ $\geq$ $C|q|$ a.e. in  $\Omega_1$.}
\end{description}
by the same argument in Theorem 1.4:
\begin{ass}
Let a bounded domain $B_2$ be a priori known. Assume the following assumptions:
\begin{description}
  \item[(i)] $q \in L^{\infty}(\Omega_1)$ with $\mathrm{Im}q  \geq 0 \ in \ \Omega_1$.
  
  \item[(ii)] $|q|$ is locally bounded below in $\Omega_1$, i.e., for every compact subset $M \subset \Omega_1$, there exists $c>0$ (depend on $M$) such that $|q|\geq c \ \mathrm{in}\ M$. 

  \item[(iii)] $\overline{\Omega_2}\subset B_2$, $\overline{\Omega_1} \cap \overline{B_2}=\emptyset.$
  
  \item[(iv')]There exists $t \in [0, \pi/2) \cup (3\pi/2, 2 \pi]$ and $C>0$ such that
$\mathrm{Re}(\mathrm{e}^{-it}q)$ $\geq$ $C|q|$ a.e. in  $\Omega_1$. 
\end{description}
\end{ass}

\begin{thm}
Let $\mathrm{Assumption \ 4.5}$ hold. Take a positive number $\lambda_0>0$. Then, for $z \in \mathbb{R}^3 \setminus  \overline{B_2}$ 
\begin{equation}
z \in \Omega_1
\Longleftrightarrow
\sum_{n=1}^{\infty}\frac{|(\phi_z,\varphi_n)_{L^2(\mathbb{S}^{2})}|^2}{\lambda_n} < \infty, \label{4.35}
\end{equation}
where $(\lambda_n,\varphi_n)$ is a complete eigensystem of $F_{\#}$ given by
\begin{equation}
F_{\#}:=\bigl|\mathrm{Re}\bigl(\mathrm{e}^{-it} F\bigr)\bigr|+\bigl|\mathrm{Im}F\bigr|, \label{4.36}
\end{equation}
where $F:=F^{Mix}_{\Omega_1q,\Omega_2}+F^{Imp}_{B_2,i\lambda_0}$.
Here, the function $\phi_z$ is given by $($$\ref{1.11}$$)$.
\end{thm}
\section*{Conclusion}
In this paper, we give the characterization of the unknown domain $\Omega_1$ in a scatterer consisting of two objects with different physical properties without the assumption of the wave number $k>0$.
To realize it, we modify the original far field operator $F$ by adding artificial far field operators corresponding to an inner domain $B_1$, an outer domain $B_2$, and an additional domain $B_3$. This idea is mainly based on \cite{Kirsch and Liu}, which treats only a scattering by an obstacle with the pure Dirichlet or Neumann boundary condition. In Section 4 of \cite{Kirsch and Liu}, numerical examples are given to compare modification method (which use the artificial far field operator corresponding to an inner domain) with previous method numerically, where we find that the modification method provides numerically a better reconstruction than previous one. Therefore, we expect that even in a scatterer consisting of two objects with different physical properties, our modification method (which use several artificial far field operators) would also provide a better reconstruction than previous ones such as \cite{Kirsch and Liu2, Liu}.

\section*{Acknowledgments}
First, the author would like to express his deep gratitude to Professor Mitsuru Sugimoto, who always supports him in his study. Secondly, he thanks to Professor Masaru Ikehata and Professor Sei Nagayasu, who read this paper carefully and gave him many helpful comments. The author also thanks to Professor Andreas Kirsch, and Professor Xiaodong Liu, and the referees who gave him valuable comments which helped to improve this paper.


Graduate School of Mathematics, Nagoya University, Furocho, Chikusa-ku, Nagoya, 464-8602, Japan \par
e-mail: takashi.furuya0101@gmail.com

\end{document}